\documentclass[smallextended]{svjour3}       
\smartqed  
\usepackage[justification=centering]{caption}
\usepackage{latexsym}
\usepackage{tkz-euclide}
\usetikzlibrary{intersections}
\usetikzlibrary{shapes,arrows,calc}
\usepackage{amssymb,amsmath}
\usepackage{geometry}
\usepackage{pgfplots}
\usetikzlibrary{calc}
\usepackage{float}
\def\doublespaced{\baselineskip=\normalbaselineskip
    \multiply\baselineskip by 2}
\def\doublespace{\doublespaced}

\newcommand{\hd}{\mbox{H-dim}\;}
\doublespace

\newcommand{\rn}[1]{\mathbb{R}^{#1}}

\newcommand{\beq}{\begin{equation}}
\newcommand{\bea}[1]{\begin{array}{#1} }
\newcommand{\eeq}{ \end{equation}}
\newcommand{\ea}{ \end{array}}
\newcommand{\ep}{\epsilon}
\newcommand{\es}{\emptyset}

\newcommand{\al}{\alpha}

\newcommand{\ga}{\gamma}

\newcommand{\de}{\delta}
\newcommand{\ds}{\displaystyle}
\newcommand{\ts}{\textstyle}
\newcommand{\rar}{\mbox{$\rightarrow$}}

\newcommand{\ran}{\rangle}
\newcommand{\lan}{\langle}
\newcommand{\Ga}{\Gamma}
\newcommand{\la}{\lambda}
\newcommand{\La}{\Lambda}
\newcommand{\ar}{\partial}
\newcommand{\si}{\sigma}

\newcommand{\om}{\omega}
\newcommand{\Om}{\Omega}
\newcommand{\be}{\beta}

\newcommand{\ph}{\phi}
\newcommand{\he}{\theta}
\newcommand{\He}{\Theta}
\newcommand{\Ph}{\Phi}
\newcommand{\hs}[1]{\mbox{$ \hspace{#1}$}}

\newcommand{\sem}{\setminus}
\newcommand{\ze}{\zeta}

\newcommand{\ti}{\tilde}

\begin{document}

\title{Hausdorff dimension and  $ \si $ finiteness of $p-$harmonic measures in space when
$p\geq n$
\thanks{Akman and Lewis were partially supported by NSF DMS-0900291}
}
\titlerunning{Hausdorff dimension and  $\si$ finiteness of $p-$harmonic measures}        

\author{Murat Akman         \and
        John Lewis 	\and
	Andrew Vogel
}

\authorrunning{M. Akman \and J. Lewis \and A. Vogel} 

\institute{Murat Akman \at
               Mathematics Department, University of Kentucky, Lexington, Kentucky, 40506\\
		\email{murat.akman@uky.edu}
           \and
           John Lewis \at
           Mathematics Department, University of Kentucky, Lexington, Kentucky, 40506\\
		\email{johnl@uky.edu}
		\and
		Andrew Vogel \at
		Department of Mathematics, Syracuse University, Syracuse, New York 13244\\
		\email{alvogel@syracuse.edu}
}

\date{June 24, 2013}

\maketitle

\begin{abstract}     
In this paper we   study   a    measure, $ \hat \mu, $  associated with a  positive  $ p $ harmonic function $
     \hat u $  defined in an open set $ O\subset \mathbb{R}^n $ and vanishing on   a portion $ \Ga $ of  $ \ar O. $  If $ p > n $ we show $ \hat  \mu $
       is concentrated on a set of  $ \sigma $ finite $ H^{n-1} $ measure  while if  $ p = n $  the  same conclusion holds provided  $ \Ga $ is  uniformly fat in the sense of  $n$  capacity.
       Our work  nearly  answers in the affirmative  a  conjecture in  \cite{L12}  and  also appears to be the  natural extension of \cite{JW88,W93} to  higher  dimensions.

\end{abstract}

\section{Introduction}
Denote points in Euclidean $ n $-space $ \rn{n}  $  by $ x = ( x_1,
   \dots,  x_n) $  and  let
       $ \bar E, \ar E, $ diam $ E, $
    be the closure,
    boundary,  and  diameter  of the set $ E \subset
    \mathbb {R}^{n}. $   Let  $ d ( E, F ) $ be the distance between the sets $ E, F $  and  $ d ( y, E ) =
    d ( \{y\}, E ). $   Let   $  \lan \cdot ,  \cdot  \ran $  denote  the standard inner
    product on $ \rn{n}  $ and let  $  | x | = \lan x, x \ran^{1/2} $ be
    the  Euclidean norm of $ x. $   Set  $ B (x, r ) =
    \{ y \in \mathbb {R}^{n} : | x  -  y | < r \} $  whenever  $ x \in \mathbb {R}^{n}, \, r
    >0, \,  $ and let  $ dx $ denote  Lebesgue $ n $-measure on    $ \rn{n}. $
       If $ O  \subset \mathbb {R}^{n} $ is open and $ 1  \leq  q  \leq  \infty, $ then by   $
    W^{1 ,q} ( O ) $ we denote the space of equivalence classes of functions
    $ f $ with distributional gradient $ \nabla f = ( f_{x_1},
     \dots, f_{x_n} ), $ both of which are $ q $
    th power integrable on $ O. $  Let  $  \| f \|_{1,q} = \| f \|_q +  \| \, | \nabla f | \, \|_{q}  \,  $
    be the  norm in $ W^{1,q} ( O ) $ where $ \| \cdot \|_q $ denotes
    the usual  Lebesgue $ q $ norm in $ O. $  Next let $ C^\infty_0 (O )$ be
     the set of infinitely differentiable functions with compact support in $
    O $ and let  $ W^{1,q}_0 ( O ) $ be the closure of $ C^\infty_0 ( O ) $
    in the norm of $ W^{1,q} ( O  ). $  If  $ K  \subset \bar{B} ( x, r  ) $ is a compact set let
     \[  C ( K,  B ( x , 2r ) )   =   \inf \int_{ \rn{n} }  | \nabla \ph |^n  \, dx   \]
     where the infimum is taken over all $ \ph \in W^{1,n}_0 ( B (x, 2r ) ) $ with
     $ \ph \equiv  1 $ on  $ K. $ We say that a compact set $ E \subset \rn{n} $ is locally $(n,r_0) $ 
uniformly fat or locally uniformly $(n,r_0)$ thick  provided  there exists
     $  r_0, \be  >  0,   $  such that whenever $ x \in E, 0 < r \leq r_0, $
       \[  C ( E \cap \bar B ( x, r ) , B ( x, 2r) ) \geq  \be.  \]

     Let    $ O \subset \rn{n}  $  be an open set  and  $ \hat z  \in \ar O. $
    Fix  $  p,  1 < p < \infty, $ and suppose that
 $ \hat  u $ is  a positive weak solution   to the $ p $ Laplace equation
in $  O\cap B ( \hat z, \rho ). $   That is,
$ \hat  u \in W^ {1,p} ( O\cap B ( \hat z, \rho ) ) $ and
\begin{equation}   
\label{1.1}  
\int | \nabla \hat  u |^{p - 2}  \, \lan  \nabla \hat  u , \nabla \he  \ran \, dx   = 0    
\end{equation}
 whenever $  \he  \in
W^{1, p}_0 (  O \cap B ( \hat z, \rho ) ). $  Equivalently we say that
  $\hat u $ is $p$ harmonic in $ O \cap  B ( \hat z, \rho ) $.
 Observe that if $ \hat  u $ is smooth and $ \nabla \hat  u \not = 0 $
in $ O \cap B ( \hat z, \rho ), $ then
$ \, \nabla \cdot ( | \nabla \hat  u |^{ p - 2} \, \nabla \hat  u ) \equiv 0,
$  in the classical sense, where $ \nabla \cdot $ denotes divergence.  We assume that  $ \hat  u $ has  zero  boundary values on
$ \ar  O \cap  B ( \hat z, \rho ) $ in the  Sobolev sense. More specifically if
$ \ze \in C^\infty_0 ( B ( \hat z, \rho ) ), $ then
$ \hat  u \, \ze \in W^{1,p}_0 (  O \cap B ( \hat z, \rho ) ). $
Extend $ \hat  u $ to  $  B ( \hat z, \rho )  $ by putting
$ \hat  u \equiv 0 $ on $ B ( \hat z, \rho )  \sem O. $   Then  $ \hat  u \in W^{1,p}
( B ( \hat z, \rho  ) )  $   and it follows from (\ref{1.1}), as in \cite[Chapter 21]{HKM06}, that
there exists a positive  Borel measure $ \hat  \mu $  on
$ \rn{n}  $ with support contained in $ \ar O \cap \bar B (  \hat z,  \rho )  $  and the property
that
\begin{equation} 
\label{1.2}
 \int | \nabla \hat  u |^{p - 2}  \, \lan  \nabla \hat  u ,
 \nabla \phi  \ran \, dx  = -  \int  \phi  \, d \hat  \mu   
\end{equation}
 whenever
 $ \ph \in C_0^\infty ( B ( \hat z, \rho  ) ). $
We note that if $ \ar O $ is smooth enough, then
$ \,  \,  d \hat  \mu = | \nabla \hat  u |^{p - 1} \, dH^{n-1} $  where $ H^{n-1} $ denotes  
Hausdorff $n-1$ dimensional measure defined after  Theorem \ref{theorem1.1}.

In this paper  we continue our study of  $  \hat \mu $  for    $    n \leq p  < \infty.  $
We prove
\begin{theorem}   
\label{theorem1.1}
Fix $p, n \leq p < \infty $  and let  $ \hat z, \rho,  \hat u,  \hat\mu $ be as in
(\ref{1.2}).  If  $ p >  n, $ then $ \hat \mu $ is concentrated on a set of  $\si $ finite  $ H^{n-1} $ measure. If $ p = n $ and  $ \ar  O \cap B ( \hat z, \rho ) $ is
locally $ (n, r_0)$ uniformly fat, then  $ \hat \mu $ is concentrated on a set of  $ \sigma $ finite $ H^{n-1} $ measure.
\end{theorem}
To define Hausdorff measure and  outline previous work we shall need some more notation.
If $ \la > 0 $ is a positive function on $ (0, \hat r_0 ) $   with
 $ { \ds \lim_{r\rar 0} \la (  r ) = 0 } $ define  $ H^\la $
Hausdorff measure on   $ \rn{n} $
as follows:  For fixed $0 <  \delta
 < \hat r_0  $ and $ E \subseteq \rn{2} $, let
$ L ( \delta ) = \{ B ( z_i,  r_i ) \} $ be such that
$ E \subseteq \bigcup \, B ( z_i , r_i ) $ and $ 0 < r_i < \delta , ~~
i = 1,2,..$.  Set
\[ \phi_{ \delta }^\la (E) = {\displaystyle \inf_{ L ( \delta ) } }
 \sum  \,  \la ( r_i )  .  \]
  Then
\[ H^\la (E) = {\displaystyle \lim_{ \delta \rightarrow 0 } }
 \, \, \phi_{ \delta }^\la (E)  . \]
In case $ \la ( r ) = r^\alpha $ we write $ H^\alpha $ for $  H^\la. $

Define the  Hausdorff dimension of a Borel measure  $ \nu  $ on $ \rn{n}$ by
\[   \hd   \nu  \, =
  \inf   \{ \alpha : \exists E
  \mbox{ Borel  with }    H^\alpha ( E ) = 0 \mbox{ and }    \nu ( \rn{n} \sem E ) = 0 \}. \]
        From Theorem \ref{theorem1.1} and the definition of  $ \hd \nu $ it is easily seen that
\begin{corollary}    
\label{corollary1.2}
Let $ \hat u, \hat \mu,  $  be as  in Theorem \ref{theorem1.1}.  Then  $ \hd \hat \mu \leq n - 1. $
\end{corollary}

For $n=2, 1 < p < \infty, $ Lewis proved in \cite{L12}
 the following theorem which generalized  earlier results in \cite{BL05,L06,LNP11}.

\begin{theorem}
\label{theorem1.3}
 Given  $ p,  1 < p  < \infty, p \not
= 2, \,  $ let  $ \hat  u, \hat  \mu  $ be as in (\ref{1.1}), (\ref{1.2}),
 with $ \rho = \infty $  and  suppose
$ O $ is a  simply connected bounded domain.
Put  
\[
\la ( r ) = \la ( r, A ) =
r \, \exp[   A  \sqrt{ \log 1/r \,  \log \log \log 1/r } ],
0 < r < 10^{- 6}.   
\]
Then the following is true.
   \[
   \begin{array}{l} 
   (a) \hs{.2in}     \mbox{  If   $ p > 2,  $    then   $ \hat \mu $ is  concentrated on a  set of  $ \si $ finite $ H^1 $ measure. }
\\
(b) \hs{.2in} \mbox{  If   $   1 < p < 2, $
then   $  \hat \mu $ is absolutely
  continuous with respect to $ H^\la$} \\ \hs{.40in} \mbox{  provided $ A = A (p) \geq 1 $ is large enough. }
    \end{array} 
    \]  
\end{theorem}

\begin{remark}
Makarov in \cite{M85} (see also \cite{GM05,M90,P92}), essentially proved Theorem \ref{theorem1.3}  for
harmonic measure, $ \om, $   with respect to a  point in
   $ O $ (the $p = 2 $ case).  Moreover,  \cite{JW88}  showed for any planar domain  whose complement is a compact set
   and for which $ \om $ exists, that
   $ \hd \om \leq 1. $   Wolff \cite{W93}  improved this result by  showing that  for any planar domain  $ \om $ is  concentrated on a set of  $ \si $ finite  $ H^1 $ measure.
\end{remark}

In higher dimensions,  $ n \geq  3, $  Bourgain \cite{B87}  showed that $ \hd  \om <  n $ for any open set $ O $ for which
   $ \om $ exists.  Building on an idea of  Carleson in \cite{C85},  Wolff in \cite{W95} constructed in $ \rn{3}, $ a  \textit{Wolff snowflake} for which $ \hd \om > 2 $ and also one for which $ \hd \om < 2. $  This was further
    generalized in \cite{LVV05} where
   it was shown that both sides of  a  Wolff snowflake in $ \rn{n}$  could have  harmonic measures, say $ \om_1, \om_2, $ with  either $ \min ( \hd \om_1, \hd \om_2 ) > n - 1 $ or  $ \max ( \hd \om_1, \hd \om_2 ) < n - 1. $

 Theorem 4 of \cite{L94}  implies for fixed $p,$   $ 1 < p < \infty, $ and $ \hat u,
 \hat \mu $ as in (\ref{1.2})  that $ \hd \hat \mu <  n - \tau $ where  $ \tau = \tau ( p, n ) > 0. $
                               Theorem \ref{theorem1.1} was proved in \cite{LNV11}
   when  $ \rho = \infty  $ and  $ O  $ is a sufficiently    flat  Reifenberg domain.  Also Wolff's method was extended to the
   $ p $ harmonic setting and produced examples of Wolff type  snowflakes and $p$ harmonic functions $ u_\infty $   vanishing on the boundary of these snowflakes for which the corresponding  measures, say  $ \mu_\infty, $  had the following  Hausdorff dimensions.
   
\begin{theorem}  
\label{theorem1.4}
If  $ p \geq n, $  then  all examples produced by Wolff's method had 
\[
\hd \mu_\infty |_{B (0,1/2)}<n-1. 
\]
Moreover for   $ p > 2,  $ near enough 2,  there existed  a \textit{Wolff snowflake} for  which
\[
 \hd \mu_\infty|_{B (0, 1/2)}>n-1.
\]   
\end{theorem}

In view of  Theorem \ref{theorem1.4} and  the above results  it is natural to conjecture that Theorem \ref{theorem1.1} 
remains valid for $ p = n $ without the uniform
fatness assumption   on  $ \ar O \cap B ( \hat z,  \rho). $   A slightly wilder conjecture is that there exists $ p_0, 2 < p_0  < n, $
such that if $ p_0 \leq p $ and  $ \hat u,  \hat \mu,  $  are the $ p $ harmonic function-corresponding  
measure as in (\ref{1.2}),  then  $ \hd \hat \mu \leq n - 1. $

As for our proof of Theorem \ref{theorem1.1}, here we first remark that it is embarrassingly simple compared to the proof in  Theorem 1$(a)$ of
   \cite{L12}.  Moreover  the main idea for the proof comes from  \cite{W93} where a simple proof for harmonic measure  in planar
    domains, whose boundaries are uniformly fat in the sense of logarithmic capacity, is outlined.  Our   proof also makes
    important use of
    work in \cite{L12} and \cite{LNV11}.   More specifically suppose for fixed $ p, 1 < p < \infty, $  that $ \hat u,  \hat \mu, O,
     \hat z, \rho $  are as in (\ref{1.2}). Then from Lemma \ref{lemma2.4} we see that $ \hat u_{x_k}, 1 \leq k \leq n, $ are  H\"{o}lder continuous in
  $ O \cap B ( \hat z, \rho ). $
  If also   $ \hat x \in O \cap B ( \hat z, \rho )$  and $ \nabla \hat u ( \hat x ) \not = 0, $  then $ \hat u $ is infinitely differentiable  in
      $B ( \hat x, \de ) $ for some $ \de > 0. $  Let $  \xi \in \ar B ( 0, 1 ) $  differentiating the $  p $  Laplace equation,
    $ \nabla \cdot ( |\nabla \hat u |^{p-2} \nabla \hat u ) = 0 $      with respect to $ \xi $ it follows that
     both $ \zeta = \hat u_\xi $ and $\zeta = \hat u, $ satisfy  the divergence
     form PDE  for $ x $ in $ B ( \hat x, \de ) $:

     \begin{equation}
     \label{1.3}
     L \ze  ( x ) \, = \,  \sum_{i,k = 1 }^n \,
    \frac{ \ar }{ \ar x_i} \, [\,   b_{i k }
    ( x )  \ze_{x_k} ( x )  \, ]  = 0,  \, 
    \end{equation}
    where  at $ x  $
    \begin{equation}
    \label{1.4}
     b_{ik} ( x )  = | \nabla \hat u |^{ p - 4} [ ( p - 2 )
    \hat u_{x_i} \hat u_{x_k}  + \de_{ik}
    | \nabla \hat u |^2 ] ( x ) ,  \, 1 \leq i,  k  \leq  n, 
    \end{equation}  
    and $
    \de_{ik} $ is
    the Kronecker $ \de. $
     From smoothness of  $ \hat u $  we see that
     $ b_{ik} $
    are infinitely differentiable in
    $ B (\hat  x, \de )  $ and   at $ x \in B ( \hat x,  \de ), $
     \begin{equation}  
     \label{1.5}   
     \min \{ p - 1, 1 \}    | \xi |^2 \, | \nabla \hat u ( x ) |^{ p -
    2} \, \leq \,  { \ds  \sum_{i,k=1}^n  \, b_{ik} \, \xi_i \xi_k }  \,
     \leq \,  \max \{ 1, p - 1 \}    | \nabla \hat u ( x ) |^{ p - 2} \, | \xi |^2
      \, . 
      \end{equation}
    The PDE in (\ref{1.3})  for  $ \hat u, \hat u_{x_k}, 1  \leq k \leq n, $ was used in  Lemma 5.1 of \cite{LNP11} 
    to show that if  $ v  =  \log | \nabla \hat u | $ and
    $ \nabla \hat u ( \hat x ) \not =  0, $ then for $ x \in   B ( \hat x, \de ), $
    \begin{equation}
    \label{1.6}  
    L \, v ( x )  \geq 0\, \,   \mbox{when}\, \,   p \geq n.  
    \end{equation}
 (\ref{1.3})-(\ref{1.6}) are  used throughout \cite{BL05,L06,LNP11,LNV11}.  Another key inequality
  in these papers  was called  the fundamental inequality:
 \begin{equation}  
 \label{1.7}
    \frac{1}{c} \,   | \nabla \hat u ( x ) |
    \leq \, \frac{\hat u( x )}{d ( x, \ar \Om)} \, \leq c | \nabla \hat u ( x ) |,   
    \end{equation}
     where $ c = c ( n, p )$. (\ref{1.7})  was shown to hold for all $ x $ near $ \ar O $ in 
     the special domains considered in Theorems \ref{theorem1.3}, \ref{theorem1.4}.  Observe that if (\ref{1.7}) 
     holds, then from (\ref{1.5})  it follows that $ L $ is locally a  uniformly elliptic operator.  
     Hence in these papers results from  elliptic PDE were used.

                                        The upper inequality in (\ref{1.7}) follows 
                                        from PDE type estimates and is true for $ O $ as in Theorem \ref{theorem1.1}.  
                                        However the lower estimate is easily seen to fail when $ \ar O $ is not connected.  
                                        Thus we  are not able to use either of the strategies in \cite{L12} or \cite{LNP11}
                                        in our proof of  Theorem \ref{theorem1.1}.  The argument in section \ref{section3} essentially uses only
           (\ref{1.3}) - (\ref{1.6}) and the basic estimates for $p$ harmonic functions in section \ref{section2}.

            As for the plan of this paper, in section \ref{section2} we list some basic estimates for $p$ harmonic functions. 
            In section \ref{section3} we use these estimates and (\ref{1.3})-(\ref{1.6}) to prove Theorem \ref{theorem1.1}.  Finally in section \ref{section4} we make closing remarks and  discuss future research.

\section{Basic Estimates  for $p$ Harmonic Functions.}
\label{section2}
In  the sequel  $ c $ will denote a  positive constant  $ \geq 1  $ (not
necessarily the same at each occurrence), which may depend only on $ p, n, $ unless otherwise stated.
In general, $ c ( a_1, \dots, a_n )$ denotes a positive constant
$ \geq 1, $  which may depend only on $ p, n, a_1, \dots, a_n$ not
necessarily the same at each occurrence.   $ A \approx B $ means
that  $  A/B $ is bounded above and below by positive constants
depending only on $ p, n. $
 In this section, we  will   always assume  that $ 2 \leq  n  \leq p < \infty , $  and $  r > 0. $  
Let $ \Om $ be an open set, $ w \in \ar \Om, $  and suppose that $ \ti u $ is $ p $ harmonic 
in $ \Om \cap B ( w, 4 r ). $ If
 $ p = n $  we also assume that $ \ar \Om \cap \bar B (w,4 r) $ is
 $(n, r_0) $ uniformly fat as defined above (\ref{1.1}).

     We begin by stating  some
interior and boundary estimates for $ \ti u, $  a positive weak solution to
the $ p $ Laplacian  in $ \Om  \cap B ( w,  4 r ) $  with
$ \ti u  \equiv 0 $ on $ \ar \Om \cap B ( w, 4r )$  in the Sobolev sense,
  as indicated after  (\ref{1.1}).    Extend $ \ti u $ to $ B ( w, 4 r ) $
by putting $ \ti u \equiv 0 $ on $ B ( w, 4 r )\sem \Om. $ Then there exists
a locally finite positive Borel measure $ \ti \mu $ with
support $ \subset \ar \Om \cap \bar B ( w, 4 r ) $  and  for which
(\ref{1.2}) holds with $ \hat u $ replaced by $ \ti u $  and $ \ph
\in C_0^\infty ( B ( w, 4 r ) ). $
  Let $ { \ds \max_{B ( z, s )}  \ti u, \,  \min_{B ( z, s)} \ti u } $ be the
 essential supremum and infimum  of $ \ti u $ on $ B ( z, s) $
whenever $ B ( z, s ) \subset B ( w, 4 r ). $
 For proofs  of Lemmas \ref{lemma2.1} - \ref{lemma2.2} (see \cite[Chapters 6 and 7]{HKM06}).

\begin{lemma}
\label{lemma2.1}
Fix $ p, 1 < p < \infty, $ and
let $\Om,  w, r, \ti u,   $ be as above.
   Then  
   \[
   \frac{1}{c}  r^{ p - n} \, \int\limits_{B ( w, r/2)}  \,  | \nabla \ti u
|^{ p } \,  dx  \, \leq \, \max_{ B ( w, r ) } \, \ti u^p  \, \, \leq
\frac{c}{r^{n}} \, \int\limits_{B ( w, 2 r ) } \, \ti u^p \, dx.  
\]  
If
$ B ( z,   2s  ) \subset \Om, $ then
\[ \max_{B ( z , s  ) } \, \ti u  \, \leq  c \min_{ B ( z , s  )} \ti u. \] \end{lemma}

\begin{lemma}   
\label{lemma2.2}
Let $ p, \Om,  w, r, \ti u,   $ be as in
Lemma \ref{lemma2.1}.  Then
there exists $ \al
= \al ( p, n )  \in (0, 1) $  such that  $ \ti u $ has a  H\"{o}lder
$ \al $ continuous representative in  $  B ( w,    4r )  $ (also denoted $ \ti u$).
 Moreover  if $ z_1, z_2 \in B ( w, r )
$  then
\[
| \ti u ( z_1 ) - \ti u ( z_2 ) | \leq \, c \,  \left( \frac{|z_1 - z_2 |}{ r }\right)^\al  \,
\max_{ B ( w, 2 r) } \,  \ti u 
\]  
\end{lemma}

\begin{lemma}
\label{lemma2.3}
Let $p, \Om,  w, r, \ti u,$ be as in Lemma \ref{lemma2.1}  and let  $ \ti \mu $  be the measure
associated with $ \ti u $ as in (\ref{1.2}).
Then there exists $ c,  \ga  = \ga (p, n ) \geq 1, $
 such that
\[
\frac{1}{c} \, r^{ p - n } \, \ti \mu [ B ( w,  r/2 ) ]  \,
\leq \, { \ds   \max_{B ( w,  r ) } \, \ti u^{ p - 1}   } \,
 \, \leq \,
 c \,    r^{ p - n } \, \ti  \mu [ B ( w,  2 r  ) ].  \, 
 \] 
\end{lemma}
For the proof of Lemma \ref{lemma2.3} see \cite{KZ03}.
The left-hand side of the above inequality is  true for any open $ \Om $ and $p \geq n. $  However the right-hand side
of this inequality  requires uniform fatness when $ p = n $ and is the main  reason  we have this assumption in Theorem \ref{theorem1.1}.
 \noindent  The reader is referred to \cite{BL05} for references concerning the proof of  the next lemma.
 
 \begin{lemma}
 \label{lemma2.4}
 Let $ p, \Om, w,  r, \ti u,$ be as in Lemma \ref{lemma2.1}.  Then $\ti u$ has a representative in
     $ W^{1,p} ( B(w,4r)) $  with H\"{o}lder
    continuous partial derivatives in $ \Om \cap B(w,4r)$. In particular, there exists $ \si   \in (0,1] $, depending only on $ p,
    n$, such that if  $ x, y \in B ( \hat w,  \hat r/2 )  $, $B ( \hat w,  4\hat r )\subset \Omega\cap B (w,4r)$, then
    \[
    \frac{1}{c} \, | \nabla  \ti u ( x ) - \nabla  \ti  u ( y ) |  \,
    \leq  \,  \left(\frac{|x-y|}{\hat r}\right)^\si \, \max_{B ( \hat w , \hat r  )} \, | \nabla   \ti u |  \,
    \leq \, \frac{c}{\hat r} \,   \left( \frac{| x  - y |}{\hat r}\right)^\si  \,
    \max_{B (\hat  w,  2\hat r  )}    \ti  u. 
    \]  
    If  $ x \in  B ( \hat  w, 4 \hat r ) $  and   $ \nabla \ti u ( x ) \not = 0, $  then  $ \ti u $ is infinitely differentiable in an open  neighborhood of  $ x. $
     Moreover,  \[ \int\limits_{ B ( \hat w , \hat r  ) \cap \{ | \nabla  \ti u | > 0 \}  } \, | \nabla  \ti  u |^{p-2}   \,
     \sum_{i,j = 1}^n \ti u_{x_i x_j}^2  dx   \leq \, \frac{c} {\hat r^{2}}
       \int\limits_{ B ( \hat w ,  2  \hat r  ) } \, | \nabla  \ti  u |^{p} \, dx. \]
       \end{lemma}

\begin{lemma}
\label{lemma2.5}
Let $ p, \Om, w,  r, \ti u,  $ be as in Lemma \ref{lemma2.1}. Suppose   for some  $ z \in \rn{n},  t \geq  100 r,    $   that  $ w \in
\ar B ( z, t ) $   and
\[
B ( w, 4r ) \sem \bar B ( z, t ) =  B ( w, 4 r ) \cap \Om. 
\]
  There exists $ \si = \si (p,n) \in (0,1) $ for which  $ \ti u |_{ \Om \cap B(w,3r)}    $ has  a  $ C^{ 1, \si } \cap W^{1,p} $  extension to   $ B ( w, 3r) $ (denoted $ \bar  u$).  If
  $ x \in  B (  w,  3r )\sem  \ar B ( z, t ) $  and   $ \nabla \bar  u ( x ) \not = 0, $
  then  $ \bar  u $ is infinitely differentiable in an open  neighborhood of  $ x. $
  Moreover,
  \[
  \int\limits_{ \Om \cap B (  w ,  r/2  ) \cap \{ | \nabla  \bar u | > 0 \}  } \, | \nabla  \bar  u |^{p-2}   \,
     \sum_{i,j = 1}^n \bar u_{x_i x_j}^2  dx   \leq \, \frac{c}{r^{2}}
       \int\limits_{ \Om \cap B (w ,  2r) } \, | \nabla  \bar  u |^{p} \, dx 
       \]
 and if    $ x, y \in  \Om  \cap B (  w,   r/2 ),  $  then
    \[
    \begin{array}{ll}  \frac{1}{c}\, | \nabla   \bar u ( x ) - \nabla  \bar  u ( y ) |  \,
       & \leq  \,  \left(\frac{| x  - y |}{ r}\right)^\si \, {\ds \max_{ \Om \cap B (  w ,  r  )} } \, | \nabla   \bar u |  \,
    \\  \\  &   \leq \, \frac{c}{r}\,   \left(\frac{| x  - y |}{ r}  \right)^\si  \, {\ds  \max_{\Om \cap B ( w, 2r )}} \bar u.
    \end{array}
    \]
\end{lemma}
\begin{proof}   
We assume as we may that $ z = 0 $ and $ t = 1 $ since otherwise we  consider  $  u^*  ( x ) = \ti u ( z  + t x ) $ and use
translation - dilation invariance of  the $p$  Laplacian. Let
\begin{displaymath}
    \bar u(x)= \left\{
     \begin{array}{cl}
       \, \, \, \ti u(x) & \mbox{when}\, \, x  \in \bar \Om \cap  B ( w, 3 r )\\
       -\ti u(\frac{x}{|x|^2} ) & \mbox{when}\, \, x  \in  B(0,1) \cap  B ( w,  3 r ).
     \end{array}
   \right.
\end{displaymath}
If  $   y  =
 x/ | x |^2    \in B(0,1) \cap B(w,3r) $ and $ \nabla  \ti u (x) \not = 0, $   one can use the chain rule to calculate  at $ y $ that 
\begin{equation}  
\label{2.1}
\nabla \cdot \left( |y|^{2p - 2n}  |  \nabla \bar u |^{p-2}   \nabla \bar u  \right)   =
 \sum_{i=1}^n  \frac{\ar}{\ar y_i}  \left(  |y|^{2p - 2n } |\nabla \bar u |^{p-2}  \frac{ \ar \bar u }{\ar y_i}  \right) = 0. 
 \end{equation}
Put
\begin{displaymath}
    \ga (x)= \left\{
     \begin{array}{cl}
       |x|^{2p-2n} & \mbox{when}\, \, |x| \leq 1\\
       1 & \mbox{when}\, \, | x | > 1.
     \end{array}
   \right.
\end{displaymath}

We assert that
$ \bar u $ is a weak solution in  $ B ( w, 3r) $  to
\begin{equation}
\label{2.2}
\nabla \cdot \left( \ga   | \nabla \bar u |^{p-2} \nabla \bar u   \right) = 0. 
\end{equation}
Indeed from the assumptions on  $ \ti u  $ we see that $ \bar u  \in  W^{1,p}  ( B ( w, 3 r ) ). $   Let  $  \ph  \in  C_0^\infty ( B ( w, 3r ) ) $ and put
\[   
\ph_1 (x)   = \frac{1}{2}  ( \ph ( x  )  -  \ph (\frac{x}{|x|^2} )   )
\]
while
\[
\ph_2 (x)   = \frac{1}{2} (  \ph ( x  )  +  \ph ( \frac{x}{|x|^2} ) ).
\]
Using the change of  variables theorem  and the knowledge garnered from  (\ref{2.1}) we see that
\[
\int\limits_{ B ( w, 4 r )}  \ga  | \nabla  \bar u |^{p-2}  \nabla \bar u \cdot  \nabla  \ph_2  \,  dx  = 0  
\]
and
\[
\int\limits_{ B ( w, 4 r )}  \ga  | \nabla  \bar u |^{p-2}  \nabla \bar u \cdot  \nabla  \ph_1   dx  = 2 \int\limits_{ \Om \cap B ( w, 4 r )}    | \nabla  \ti  u |^{p-2}  \nabla \ti  u \cdot  \nabla  \ph_1  \, dx  = 0  
\]
Since  $ \ph = \ph_1 + \ph_2, $  we conclude from the  above displays that $ \bar u $ is a weak solution  to  (\ref{2.2}) in $ B ( w, 3 r ). $

 From our  assertion   we  see that  $ \bar u $ satisfies the hypotheses in \cite{T84},  except  for  $ \ga $ being continuously differentiable.  However the argument in
 \cite{T84}  and  all constants use only  Lipschitzness of  $ \ga, $  so is also valid in our situation.    Applying   the results in \cite{T84} (similar to Lemma \ref{lemma2.4}) and using  the definition
 of $  \bar u, $    we  obtain  the  first and second  displays in Lemma  \ref{lemma2.5}.
 \qed
 \end{proof}
 \begin{lemma}
 \label{lemma2.6}
Let $ p, \Om, w,  r, \ti u,  $ be as in Lemma \ref{lemma2.1} and $ - \infty < \eta \leq - 1. $ Let  $ L, (b_{ik})  $  be
as in  (\ref{1.3}), (\ref{1.4}), when  $ x  \in  \Om \cap B ( w, 4 r ) $ and $ \nabla \ti u ( x ) \not = 0. $ Let
$ b_{ij}  =  \de_{ij} $ when $ \nabla \ti u ( x ) = 0 $
 and put   $ v =   \max \{ \log |\nabla  \ti u |, \eta \}. $
Then $ v $ is locally  a weak sub solution to  $ L $ in  $ \Om \cap B ( w, 4 r ). $
\end{lemma}
\begin{proof} 
From Lemma \ref{lemma2.4}
we see that  $ v  $ is locally in  $ W^{1,2} ( \Om \cap B ( w, 4r ) ). $  Given $ \ep, \de, \si > 0 $ small
define  $ g $  by
\[
g (x)  =   ( \max \{  v - \eta - \ep, 0  \} + \si  )^\de  - \si^\de, x \in  \Om \cap B ( w, 4 r). 
\]
   As  mentioned earlier  in  Lemma 5.1 of \cite{LNP11} we showed that
$ L v  \geq 0 $ at $ x \in  \Om \cap B ( w, 4 r ) $  when $ v ( x ) \not = \eta. $ For the reader's convenience we repeat  
this calculation after the proof of Lemma \ref{lemma2.6}.

  From this fact
we deduce that if   $0 \leq  \he \in C_0^\infty ( \Om \cap B ( w, 4r ) ), $ then
\[
\begin{array}{ll}0 \leq  { \ds  \int\limits_{ \Om \cap B ( w, 4r ) } \he  g  L v dx } =
	& -  {\ds  \sum_{i,k=1}^n  \, \, \int\limits_{ \Om \cap B ( w, 4r ) } b_{ik}  ( \he g )_{x_i}  v_{x_k} dx}
\\ \\  & \leq \, - { \ds \sum_{i,k=1}^n\, \,  \int\limits_{ \Om \cap B ( w, 4r ) } } g   b_{ik}   \he_{x_i}  v_{x_k} dx,
\end{array} 
\]
where in the last inequality we have used  (\ref{1.5}). Using the above inequality,  the bounded convergence theorem,
and letting first $ \ep, $   second $ \si, $ and third $ \de  \rar 0, $ we get Lemma \ref{lemma2.6}.   
\qed
\end{proof}

To  show  $ L v (x)  \geq 0  $ when $ v ( x ) \not = \eta, $ put $ \tau ( x ) = 2 v ( x ) =  \log | \nabla \ti u |^2. $ We calculate at $x,$
\[
     \tau_{x_j} =  \sum_{k=1}^n  \, \frac{ 2 \ti u_{x_k}\ti u_{x_kx_j}}{|\nabla \ti u|^2}
       .  \]
Furthermore,
    \begin{eqnarray*}
    L \tau &=&  \sum_{i, j, k = 1}^n \, \left(b_{ij}  \dfrac{ 2
    \ti u_{x_k}\ti u_{x_kx_j}}{|\nabla \ti u|^2}\right)_{x_i}\notag\\
    &=&
     \sum_{i, j, k = 1}^n  \dfrac{2\ti u_{x_k}}{|\nabla  \ti u|^2} \left(b_{ij}
    \ti u_{x_kx_j}\right)_{x_i}   +
    \sum_{i,j,k=1}^n   2 b_{ij} \, \ti u_{x_kx_j}  \left(\dfrac{\ti u_{ x_k}}{|\nabla \ti u|^2}
    \right)_{x_i}.
    \end{eqnarray*}
    The first term on the right is zero since $L\ti u_{x_k} =0$ (see (\ref{1.3})).  We differentiate the second
    term to get
     \begin{equation}
     \label{2.3}
    L\tau  =
    \sum_{i, j, k = 1}^n    \left[ \, 2 \, |\nabla \ti u|^{-2}   b_{ij} \ti u_{x_kx_j} \ti u_{x_kx_i}
    - \sum_{i,j,k,l = 1}^n  \, 4 |\nabla \ti u|^{-4}  \, \ti u_{x_k}\ti u_{x_kx_j}b_{ij}\ti u_{x_l}\ti u_{x_lx_i}
    \right].
     \end{equation}     
We assume  as we may that  $ \ti u_{x_j} = 0 $ for $ j \not = 1, $ since otherwise we rotate our coordinate system and use invariance of the  $p$ Laplace equation under rotations.  Under this assumption
    we have
    \begin{align*}
    b_{11} & =  ( p - 1 )  \, |\nabla \ti u|^{p-2},  \\
    b_{ii} & = |\nabla \ti u|^{p-2} \quad i\neq 1, \\
    b_{ij} & =0 \quad i\neq j.
    \end{align*}
    Using these  equalities in (\ref{2.3})  we obtain, at $x,$
    \[
    L \tau  = 2 |\nabla \ti u|^{p-4} \left( (p-1) \sum_{k=1}^n \ti u_{x_kx_1}^2 +
    \sum_{i = 2,  k = 1}^n  \ti u_{x_kx_i}^2 -2(p-1)
    \ti u_{x_1x_1}^2-\sum_{i= 2}^n 2 u_{x_1x_i}^2 \right).
   \]
    Collecting the $x_1x_1$ and $x_1x_i$ ($i\neq 1$) derivatives
    yields  
    \begin{equation} 
    \label{2.4}
    L \tau  = 2 |\nabla \ti u|^{p-4} \left(-(p-1)\ti u_{x_1x_1}^2 +(p- 2)
    \sum_{k =  2}^n \ti u_{x_kx_1}^2   + \sum_{k, i =2}^n \ti u_{x_kx_i}^2 \right).
     \end{equation}
    The last sum contains the pure second derivatives of $\ti u$ in the
    $ x_k $ direction when $ k \not = 1. $  These derivatives may be estimated
    using the $p$-Laplace equation for $u$ at the point $x$, i.e., at $x$ we have
    \[
    (p-1)\ti u_{x_1x_1} + \sum_{k= 2}^n \ti u_{x_kx_k} =0.
    \]
    Solving for $\ti u_{x_1x_1}$, taking squares and using H\"older's inequality we see that 
    \[
    \sum_{k= 2}^n \ti u_{x_kx_k}^2 \geq \dfrac{(p-1)^2}{n-1}\ti u_{x_1x_1}^2.
    \]
    Substituting this expression into (\ref{2.4})   gives
    \[
    L \tau  \geq 2 |\nabla \ti u|^{p-4} \left(
    ( {\ts \frac{(p-1)^2}{n-1}} - (p-1) )\ti u_{x_1x_1}^2 +(p-2) \sum_{k=
    2 } \ti u_{x_kx_1}^2   + \sum_{k, i  = 2,   k\neq i}^n \ti u_{x_kx_i}^2 \right).
    \]
    Thus,  $L \tau  \geq 0$ when $\frac{(p-1)^2}{n-1}-(p-1) =
    \frac{(p-1)(p-n)}{n-1} \geq 0$. In particular, $L \tau  \geq 0$ if $p\geq n$.   Note that when $p=n$
then $ \ti u ( x ) = \log |x| $ is $n$ harmonic and  $ L ( \log | \nabla \ti u |   )  \equiv 0 $ when $ x \not = 0 $.

\section{Proof of  Theorem \ref{theorem1.1}.}
\label{section3}
Let $ p, n, O,  \hat u, \hat \mu,  \rho,  \hat z, $  be as in  Theorem \ref{theorem1.1} and suppose that $ \la $ is a positive nondecreasing
function on  $(0,1]$  with  $  \lim_{ t \rar 0}  t^{1-n} \la ( t )  = 0. $
  Theorem \ref{theorem1.1}  follows easily from the next proposition(See section \ref{subsection3.2}).
\subsection{Proof of Proposition  \ref{proposition3.1}}
\begin{proposition}
\label{proposition3.1}
There  exists  $ c = c (p, n ) $  and a set  $ Q \subset \ar O \cap  B ( \hat z, \rho ) $   with the following properties.
 $  \hat \mu ( \ar O \cap B ( \hat z, \rho ) \sem Q ) = 0$
 and for every  $ w  \in Q  $  there are arbitrarily small $ r = r (w), 0  < r \leq  10^{-10},  $  such that
  \[
  \begin{array}{lc}
   (a) & \bar B ( w, 100 r ) \subset B ( \hat z, \rho ) \, \, \mbox{ and }\, \, \hat  \mu ( B ( w, 100 r ) ) \leq  c\,   \hat \mu ( B ( w, r ) ).
\end{array}
\]
 Moreover there is a  compact  set  $  F = F ( w, r ) \subset   \ar O \cap B ( w, 20 r ) $ with
  \[
  \begin{array}{lc}
   (b) &  H^\la ( F ) = 0 \, \, \mbox{  and }\, \,  \hat  \mu ( F )  \geq \frac{1}{c}\, \hat \mu  ( B ( w, 100 r )).\, \, \, \, \, \, \, \, \, \, \, \, \, \, \, \, \, \, \, \, \, \, \, \, \, \, \, \, \, \, \, \, \, \, \, \, \, \, \, \,
  \end{array}
  \]
  \end{proposition}
 \begin{proof}   
To prove  $(a) $ of  Proposition \ref{proposition3.1}  we note that  $ \hat \mu (B ( x, t ))  \not = 0 $ whenever
 $ x \in \ar O$ and $\ar O \cap B ( x, t )   \subset \ar O  \cap  B ( \hat z, \rho ) $ and $ t > 0 $ as follows from
Lemma \ref{lemma2.3}.    Let 

\[   
\He =
\left 
\{ x \in \ar O \cap B ( \hat z, \rho )   :  \liminf\limits_{ t \rar 0 }\, \,   \frac{ \hat \mu ( B ( x,100 t
) ) }{\hat \mu ( B ( x, t ) )}  \geq  c  
\right\} 
\]

If $ x \in  \He,  $  then  there exists $
t_0 ( x ) > 0 $  for which

\[
 \hat \mu ( B ( x, 100 t ) ) \geq \frac{c}{2}  \hat \mu ( B ( x,
  t ) ) \mbox{ for }  0 < t <  t_0 ( x ). 
\]
Iterating  this  inequality  it follows
that  if  $ c $ is large enough then

\[
  \lim_{t \rar 0 } \frac{\hat  \mu ( B ( x, t )
)  }{t^{n+1} } = 0  \mbox{  whenever  $ x \in  \He .$}  
\]

Since    $  H^{n+1}  ( \rn{n} ) = 0, $  we  conclude that    $ \hat \mu ( \He ) =   0. $
   Thus we assume  $(a)$   holds  for some  $ c' =  c' (n),  w  \in \ar O \cap B ( \hat z, \rho ), $  and  $ r > 0. $

    To prove $(b)$  of  Proposition \ref{proposition3.1}
  let
\[
  \ga^{-1} =  {\ds \max_{B(w, 10r )}  \, \hat u}
\]
and put
\[
 u ( x ) =  \ga  \hat u ( w +  r x )\, \, \mbox{when}\, \,  w +  r  x \in   B ( \hat z, \rho ).
\]
Let
\[
\Om =  \{ x : w +   r  x  \in  O \cap B ( \hat z , \rho) \}.
\]
Using translation and dilation invariance of  the $ p$  Laplacian  we find  that  $ u $ is  $  p $  harmonic in   $ \Om  $  and if  $ \ze  = r^{-1} ( \hat z - w ), $  then  $ u $ is continuous in  $ B ( \ze, \rho/r ) $ with $ u \equiv 0 $ on  $ B ( \ze, \rho/r) \sem \Om. $
                              Moreover there exists a   measure $ \mu $  on $ \rn{n} $ with support in  $ \ar \Om
                              \cap \bar B ( \ze, \rho/r ) $ corresponding to $ u. $  In fact if
     $ E $ is a  Borel set  and $ T (  E ) =  \{ w +  r x : x \in E \} $  then
     $ \mu ( E ) $  =  $  r^{p-n}  \ga^{p-1}  \hat \mu ( T (E)). $    From   Lemma \ref{lemma2.3}  and Proposition \ref{proposition3.1} $(a),$  we  obtain  for some $  c =   c (p, n ) \geq 1 $   and
      $  2 \leq  t  \leq  50  $    that
\begin{equation}
\label{3.1}  
\frac{1}{c}  \leq  \mu ( B (0, 1) ) \leq \max_{B(0, 2 )} u  \leq  \max_{B (0, t )} u  \leq   c\,  \mu ( B (0, 100)) \leq  c^2  . 
\end{equation} 
From  (\ref{3.1})  and the definition of $ u $  we observe that to prove Proposition \ref{proposition3.1} $(b)$  it suffices to show that
     there exists a compact set  $ F'  \subset  B ( 0, 20) $  and  $ \hat c =  \hat c (p, n ) \geq 1 $ with
     \begin{equation}  
     \label{3.2}   
     \mu ( F' ) \geq   \frac{1}{\hat c}\, \,   \mbox{ and }\, \,    H^\la ( F' ) = 0.  
     \end{equation}
     To prove (\ref{3.2}) we first show for given
     $ \ep, \tau > 0 $ that there exists a  Borel set  $ E \subset  B (0, 20) $ and $ c = c (p,n) \geq 1 $ with
     \begin{equation}
     \label{3.3}  
     \phi^{\la}_\tau ( E ) \leq \ep\, \, \mbox{ and } \, \, \mu ( E ) \geq  \frac{1}{c}. 
     \end{equation}
     (\ref{3.2}) follows easily from  (\ref{3.3}).  Indeed,  choose  $  E_m  $ relative to  $ \tau = \ep = 2^{-m}, m = 1, 2,  \dots $  and put
     \[   
     E   =  \bigcap_k   \left(  \bigcup_{ m = k } E_m  \right). 
     \]
     Then from measure theoretic arguments it follows that
      (\ref{3.2}) is valid with $  F' $   replaced by  $ E $ and $ \hat c $ by $ c''. $    Using regularity of  $ \mu  $ we then  get  (\ref{3.2})  for  a  compact set
       $ F'  \subset  E. $   Thus to complete the proof of  Proposition \ref{proposition3.1}  we need only prove  (\ref{3.3}).

         To prove (\ref{3.3})  we note from  the definition of $u$   that
      $ u ( \ti z ) = 1 $ for some  $ \ti z  \in \ar B ( 0, 10  ). $  This note, (\ref{3.1}),  and Lemma \ref{lemma2.2}  imply  for some $ c_- = c_- ( p, n ) \geq 1 $ that
      \begin{equation}
      \label{3.4}    
      d ( \ti z,  \ar \Om ) \geq \frac{1}{c_-} \, . 
      \end{equation}  
      In fact  otherwise it would  follow from Lemma \ref{lemma2.2} that
      $ \max_{B (0, 20) }   u   $ is too large  for  (\ref{3.1}) to hold.

       Next let  $ M $  be a large positive number and
       $0  <   s <  e^{ -M } . $   For the moment we allow  $ M $ to vary but shall later fix it to satisfy several conditions.  We then choose
       $ s = s ( M ). $  First   given $ 0<\ti \tau < \min ( \tau, 10^{-5} )  $ choose  $  M  $ so large that if
              \begin{equation} 
              \label{3.5}
              z  \in \ar \Om \cap \bar B (0, 15)\, \, \mbox{ and }\, \, \mu ( B ( z,  t ) ) = M t^{n-1}
              \mbox{ for some $  t = t ( z )  \leq 1, $ then $ t \leq  \ti \tau.$}   
              \end{equation}
    Existence of  $ 1 \leq   M = M ( \ti \tau  ) $ follows from  (\ref{3.1}).  Next following Wolff \cite{W93} we  observe from (\ref{3.5}) that for
    each $ z \in \ar \Om \cap \bar B (0, 15) $ there exists a largest $ t = t (z)$, $s
    \leq  t \leq \ti \tau,    $  with
     either   
     \begin{equation}  
     \label{3.6}   
     \begin{array}{l}
     (\al)  \hs{.2in}   \mu (B (z, t))  =  M  t^{n-1}, t > s,  \\  \mbox{ or}  \\
     (\be) \hs{.2in}   t =  s.     
     \end{array}  
     \end{equation} 
     Using the  Besicovitch covering theorem
     (see \cite{Ma95}) we  now  obtain  a  covering  $\{ B ( z_j, t_j ) \}_1^N $
     of   $ \ar \Om \cap \bar B( 0, 15 ), $     where
       $ t_j = t( z_j )  $  is the maximal $ t $ for which either  (\ref{3.6}) $ (\al) $ or  $ (\be ) $  holds.  Moreover
       each point of   $ \bigcup_{j=1}^N  B ( z_j, t_j ) $   lies in at most  $ c = c (n) $  of  $\{ B ( z_j, t_j ) \}_1^N . $
             Let $ c_- , \ti z, $ be as in   (\ref{3.4})  and set  $ r_1 = (8 c_-)^{-1}. $  
             Choosing $ \ti \tau $ smaller (so $M$ larger)  if necessary  we may assume, thanks to (\ref{3.5}),  that
          \begin{equation}
          \label{3.7}
          \bigcup_{j=1}^N  \bar B ( z_j, 6 t_j ) \cap B ( \ti z, 6r_1 ) = \es.   
          \end{equation}           
           Also put        
           \[
           \Om'  =  \Om \cap B (0, 15 ) \sem  \bigcup_{j=1}^N  \bar B ( z_j, t_j )
           \]
           and
              \[
              D  =  \Om' \sem \bar B (\ti z, 2r_1).
              \]

\begin{figure}[h!]
\begin{center}
\begin{tikzpicture}
\coordinate (1) at (0,0);
\draw (1) circle (.2pt);
\coordinate (2) at (-3,0);
\coordinate (3) at (-2.5,0);
\draw (3) circle (.2pt);
\coordinate (4) at (1.5,0);
\draw (4) circle (.2pt);
\coordinate (5) at (3,0);
\draw (5) circle (.2pt);
\coordinate (6) at (2,2);
\draw (6) circle (.2pt);
\coordinate (7) at (2.3,-1.5);
\draw (7) circle (.2pt);
\coordinate (8) at (-1.8,-1.2);
\draw (8) circle (.2pt);
\coordinate (10) at (0,1);
\draw (10) circle (.2pt);
\coordinate (11) at (-2.5,2);
\coordinate (9) at (0.6,-2.8);
\draw[gray, fill, opacity=0.2] (1) circle (4cm);
\draw[white,fill, opacity=0.8] (2) circle (.4cm);
\draw[dashed, white, fill, opacity=0.8] (3) circle (.4cm);
\draw[dashed, white, fill, opacity=0.8] (4) circle (.3cm);
\draw[dashed, white, fill, opacity=0.8] (5) circle (.3cm);
\draw[dashed, white, fill, opacity=0.8] (6) circle (.45cm);
\draw[dashed, white, fill, opacity=0.8] (7) circle (.2cm);
\draw[dashed, white, fill, opacity=0.8] (8) circle (.4cm);
\draw[gray, fill, opacity=0.2] (9) circle (.4cm);
\draw[dashed, white, fill, opacity=0.8] (10) circle (.3cm);
\draw[dashed, white, fill, opacity=0.8] (11) circle (.5cm);
\draw[dashed, white, fill, opacity=0.8] (1) circle (.3cm);
\draw (2) circle (.2pt) node[below]{{ $z_{1}$}};
\draw (6) circle (.2pt) node[below]{{ $z_{N}$}};
\draw (9) circle (.2pt) node[below]{$ \tilde{z}$};
\draw (1)--node[midway, left]{{\tiny radius=15}}(0,-4);
\draw (0,0)--node[midway, right]{{\tiny radius=10}}(9);
\end{tikzpicture}
\label{figure2}
\caption{An example of $\Om'  =  \Om \cap B (0, 15 ) \sem  \bigcup_{j=1}^N  \bar B ( z_j, t_j )$.}
\end{center}
\end{figure}
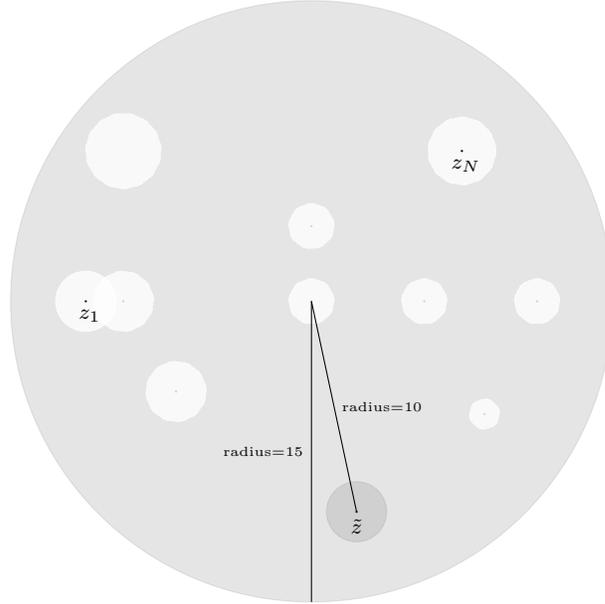
              Let
        $   u' $ be the $ p $ harmonic function in
       $ D  $  with  continuous boundary values,
      \begin{displaymath}
       u'(x)\equiv\left\{
       \begin{array}{cl}
        0& \mbox{when}\, \, x\in \ar \Om'\\
        \min\limits_{\bar B(\ti z,  2r_1)} u &\mbox{when} \, \, x\in\ar B ( \ti z, 2 r_1 ).
       \end{array}
       \right.
      \end{displaymath}
Extend $ u' $ continuously to $ \bar B (0, 15)  $ (also denoted $ u' $) by putting
\begin{displaymath}
 u'(x)\equiv\left\{
 \begin{array}{cl}
  0& \mbox{when}\, \, x\in\bar B ( 0, 15) \sem \Om'\\
        \min\limits_{\bar B(\ti z,  2r_1)} u &\mbox{when} \, \, x\in\bar B ( \ti z, 2 r_1 ).
 \end{array}
 \right.
\end{displaymath}
We  note that  $ u'  \leq u $ on  $  \ar D $  so by the  maximum principle for $ p $ harmonic 
functions  $ u' \leq u $ in  $ D. $   Also,  $ \ar D  $ is locally  $(n,r'_0)$  uniformly fat where $ r_0' $ 
depends only on  $ p, n, $ and $ r_0 $ in Theorem \ref{theorem1.1}.

To  continue the proof of   (\ref{3.3})  we shall need several lemmas.
       \begin{lemma}
       \label{lemma3.2}
       If  $ x  \in D, $  then
         \[ 
         | \nabla u' (x)   |    \leq   c\,  M^{ \frac{1}{p-1}} \, . 
         \]
         \end{lemma}
     \begin{proof}   
     To prove  Lemma \ref{lemma3.2}  let  $ x \in  D $ and choose $ y \in \ar D $ with  $  | x - y |  =  d ( x, \ar D ) = d.$   If    $ y   \in \ar  B ( z_k, t_k ) $  and  $ x  \in
      B (z_k, 2 t_k )  $ we put
       \[  f ( w )   =  A \left( | w - z_k |^{ \frac{ p - n}{ p - 1}}  -t_k^{ \frac{ p - n}{ p - 1}} \right) ,  w   \in   B ( z_k, 2  t_k ) \sem \bar B ( z_k, t_k ), \]
       when $ p > n $ and
       \[  f ( w )   =  A  \left(\log | w - z_k |  -\log t_k\right),      w   \in   B ( z_k, 2  t_k ) \sem \bar B ( z_k, t_k ) \]
       when $ p = n. $
Then  $ f \equiv 0 \mbox{   on  }  \ar B ( z_k, t_k ) $  and  $ A$ is chosen so that \[   f \equiv  \max_{B ( z_k, 2 t_k ) }  u
       \mbox{ on }   \ar B ( z_k, 2 t_k ). \]
                Then from $ u'  \leq u  $  and  the maximum  principle for  $p$  harmonic functions,   $ u' \leq   f  $  in $ B ( z_k, 2  t_k ) \sem \bar B ( z_k,  t_k ). $
        Using  this inequality and  applying Lemma \ref{lemma2.4} to  $u' $  we conclude that
         \begin{equation}
         \label{3.8}   
         | \nabla u' (x)   |   \leq  \frac{c}{d} \, \, u'(x)  \, \leq \frac{c}{d}\, \,  f(x) \leq  \, \frac{c^2}{t_k} \, \max_{B(z_k, 2 t_k)}  u  . 
         \end{equation}
        Also from  Lemma \ref{lemma2.3} and  (\ref{3.5})-(\ref{3.7})   we find that
\begin{equation}
\label{3.9}
t_k^{1-p}    \max_{B(z_k, 2 t_k)}  u^{p-1}  \leq  c\,  t_k^{1-n}\,  \mu ( B (z_k, 4 t_k) ) \leq c^2   M. 
\end{equation}
       Taking  $ 1/(p-1) $  powers of  both sides of (\ref{3.9}) and  using the  resulting inequality in  (\ref{3.8})  we get   Lemma \ref{lemma3.2}
       when  $ y  \in  \ar B ( z_k, t_k ) $ and $ x \in D \cap B ( z_k, 2 t_{k} ). $     If $ y  \in  \ar B (0, 15) $  or  $ \ar B ( \tilde z , 2 r_1 ) $  a  similar argument applies.
       Thus there is  an open   neighborhood, say  $ W, $  containing $  \ar D $  for which  the conclusion of  
       Lemma \ref{lemma3.2} is valid  when $ x \in W \cap  D. $
From  this  conclusion, Lemma \ref{lemma2.6}  applied to $u',$ and  a  maximum  principle  for weak  sub solutions to $ L, $  
we conclude that Lemma \ref{lemma3.2}  is valid in $ D. $ 
\qed
\end{proof}
              Next we prove
\begin{lemma} 
\label{lemma3.3}
The functions $| \nabla u' |^{p-2}\,  |u'_{x_k x_i}| $ for $1\leq i,k\leq n$ are all integrable on $D$
\[
{\ds \sum_{i,k=1}^n  \, \int\limits_{D}}| \nabla u' |^{p-2}\,  |u'_{x_k x_i}| \,  dx  \,  <  \,   \infty 
\] 
\end{lemma}
\begin{proof}
         Let  $ \La  \subset \ar \Om' $ be the set of points where  $ \ar \Om' $ is not smooth.   
         Clearly  $ H^{n-1} (\La) = 0 .$
        If   $ \hat x \in \ar  D \sem\La, $ then   $ \hat x $ lies in exactly one of 
        the finite number of spheres which contain points of  $ \ar D. $ Let  $ d' ( \hat x ) $ denote the distance from
        $ \hat x $ to the union of spheres not containing $ \hat x $ 
        but containing points of  $ \ar D. $  If $ d'  = d' ( \hat x ) < s/100, $ then from Lemma \ref{lemma2.5} applied to
        $ u' $  we see that each component of
         $ \nabla u'  $ has a  H\"{o}lder  continuous extension to  $ B ( \hat x,  3 d'/4 ). $   
         Also from H\"older, Lemma \ref{lemma2.5},  and   Lemma \ref{lemma3.2} we see that
      \begin{equation}
      \label{3.10}
      \begin{aligned}
  \frac{1}{c} {\ds \sum_{i, k =1}^n\, \,   \int\limits_{ D \cap B ( \hat x, \frac{d'}{8} )}  } \, | \nabla u' |^{p-2} \,  | u'_{x_i x_k} | \, dx
          &\leq (d')^{\frac{n}{2}}    M^{\frac{p-2}{2(p-1)}}  {\ds  \sum_{i, k =1}^n\, \,  \left(  \int\limits_{
         D \cap B ( \hat x, \frac{d'}{8} )}  \, | \nabla u' |^{p-2}  \,  | u'_{x_i x_k} |^2  \, dx \right)^{\frac{1}{2}} }\\
         & \leq   c (d')^{\frac{(n - 2)}{2}}  \,   M^{\frac{p-2}{2(p-1)}}
            {\ds  \left(  \int\limits_{ D \cap B ( \hat x, \frac{d'}{2} )}  \, | \nabla u' |^{p}  \,  dx \right)^{\frac{1}{2}} }
          \\&  \leq  \, c^2
              M \,  (d')^{(n-1)}.
              \end{aligned}
              \end{equation}
              To  prove  Lemma \ref{lemma3.3}  we assume as we may that $ B ( z_l, t_l ) \not \subset  B ( z_\nu, t_\nu ) $ when $ \nu \not = l, $ since otherwise we discard one of  these balls.   Also from   a  well known  covering theorem we get  a  covering  $  \{ B (y_j, \frac{1}{20} d' (y_j) ) \}  $  of  $ \ar D \sem \La $  with
   $  \{ B (y_j, \frac{1}{100} d' (y_j) ) \}, $ pairwise disjoint.   From   (\ref{3.10}) we  find that
          \begin{equation}  
          \label{3.11}
          \begin{array}{ll} 
             {\ds \sum_{i,j,k}\, \,  \int\limits_{ D \cap B ( y_j, \frac{1}{8} d'(y_j) ) }| \nabla u' |^{p-2} |u'_{x_k x_i}|  dx } \,
          & \leq c \,   M {\ds  \sum_{j} } (d' (y_j))^{n-1}   \\ \\ &   \leq c^2 M H^{n-1} ( \ar D ). 
          \end{array}  
          \end{equation}

           For short we now write
          $ d ( x ) $ for $ d ( x, \ar D) $ and  choose a  covering $ \{  B ( x_m, \frac{1}{2} d ( x_m )\} $ of  $ D $ with
   $ \{  B ( x_m, \frac{1}{20} d ( x_m )\}, $ pairwise disjoint.
   We note that if  $ x  \in D $ and $ y \in \ar D $ with  $ | y - x | = d ( x ), $  then  $ y \in \ar D \sem \La. $
Indeed otherwise $ y $ would be on the boundary of  at least two balls contained in  the complement of  $ D $ and so
by the no containment assumption above, would have to intersect $ B ( x, d (x) ),$ which clearly is  a  contradiction.
Also  we note that if  $ d ( x ) \leq  1000 s, $  then  $ d ( x ) \leq   \kappa \,  d' (y) $  where  $ \kappa  $  can depend 
on various  quantities  including the configuration of the $B(z_k,t_k)$ balls  but is  independent of
$  x \in D $ with  $ d ( x )  \leq 1000 s. $    Indeed  from the no containment  assumption one just  needs  to  consider
$ d ( x )/d'(y) $ as  $  d(x), d'(y)    \rar  0.  $  To do this  suppose  $ z \in  \La $ with  $ | y - z | = d'(y). $ 
Then one  sees,  from consideration of  half planes containing $ z $  and  tangent to  two intersecting spheres,  
that  $x, y $ eventually lie in a  truncated cone
of height  $ \ga $ with vertex at  $ z, $ and of  angle opening  $ \leq  \al <  \pi/2, $ where $ \al, \ga  $  are 
independent of $ x, y, z. $  Moreover the complement of  this truncated cone in a   certain  hemisphere of  
radius $ \ga $   with center $ z $   lies  outside of
$  \Om' $.  Then  a  ballpark estimate using trigonometry  gives  $  d' (y ) \geq  ( 1 - \sin \al)  d ( x)$(See Figure
\ref{figure1}).
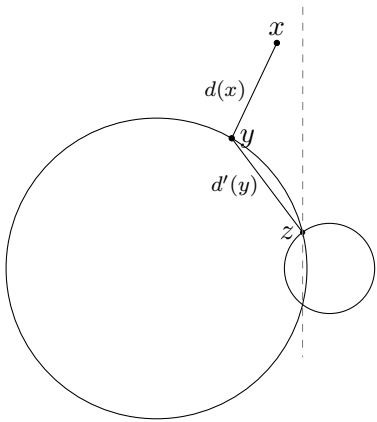
\begin{figure}[h!]
\begin{center}
\begin{tikzpicture}
\coordinate (1) at (0,0);
\draw[name path=circle 1] (1) circle (2cm);
\coordinate (2) at (2.3,0);
\draw[name path=circle 2] (2) circle (0.6cm);
\fill[name intersections={of=circle 1 and circle 2,total=\t}]
                    (intersection-1) circle (1pt) node[left]{\large{$z$}};
\coordinate (3) at (1, 1.732050808);
\draw[fill] (3) circle (1pt) node[right]{\large{$y$}};
\coordinate (4) at (1.6,3);
\draw[fill] (4) circle (1pt) node[above]{\large{$x$}};
\draw (3)-- node[midway, left] {$d(x)$} (4);
\draw (3)--node[midway,left] {$d'(y)$}(intersection-1);
\draw[shorten >=-0.7cm, shorten <=-3cm, dashed, gray] (intersection-1)--(intersection-2);
\end{tikzpicture}
\caption{$  d' (y ) \geq  ( 1 - \sin \al)  d ( x)$.}
\label{figure1}
\end{center}
\end{figure}

From this analysis and our choice of covering of $D$ we see that  for a given $ B ( x_m, \frac{1}{2}d ( x_m ) ) $  with $ d ( x_m ) <  1000 s, $
there exists $ j = j ( m )$   with  $ B ( x_m,\frac{1}{2} d ( x_m ) ) \subset  B ( y_j , \kappa' d' (y_j) ) $ for some $ 0< \kappa' < \infty $ independent of  $ m. $

Let  $ S_l, l = 1, 2, 3, $ be disjoint sets of  integers defined as follows.
\begin{displaymath}
 \left\{
 \begin{array}{ll}
  m \in  S_1 & \mbox{if}\, \, \, \,  d ( x_m  ) \geq  1000 s,\\
  m \in S_2 & \mbox{if}\, \, \, \, m \not \in S_1\, \, \mbox{and there  does not exist}\, \, j \, \, \mbox{with}\, \,  B ( x_m, \frac{1}{2} d (x_m))  \subset  B ( y_j, \frac{1}{8} d' (y_j) ),\\
 m\in S_{3}& \mbox{if} \, \, \, \, m\, \, \mbox{not in either}\, \,   S_1\, \,  \mbox{or}\, \, S_2.
 \end{array}
 \right.
\end{displaymath}
        Let
        \[
        K_l  =   \sum_{m\in S_l}\, \, \int\limits_{ D \cap B ( x_m, \frac{1}{2} d(x_m) ) } | \nabla u' |^{p-2} |u'_{x_k x_i}|  dx \, \, \mbox{for}\, \,   l =1, 2, 3. 
        \]
                Then
                 \begin{equation}
                 \label{3.12}  
                 {\ds \int\limits_{D}}| \nabla u' |^{p-2} |u'_{x_k x_i}|  dx  \,  \leq  \,  K_1 \,  + \, K_2 \,  + \, K_3  .
                   \end{equation}
From Lemma \ref{lemma2.4} and the same argument as in (\ref{3.10}) we see that
   \begin{equation}
   \label{3.13}
   K_1  \leq  c\, M \sum_{m\in S_1}  d(x_m)^{n-1}
       \leq  c^2 M  s^{-1} 
       \end{equation}
       where we have used disjointness of our covering,  $ \{ B ( x_m, \frac{1}{20} d (x_m) ) \} $ .
Using disjointness of  these balls  and  (\ref{3.11})  we get
\begin{equation}  
\label{3.14}
K_3  \leq  c\, M  H^{ n-1 } (\ar D).  
\end{equation}
 Finally   if  $ m \in S_2, $ then as discussed earlier  there exists  $ j  =  j ( m ) $  with  $ d ( x_m ) \approx  d' (y_j), $   
 where  proportionality constants are independent of  $ m, $ so
       $  B ( x_m, \frac{1}{2} d ( x_m ) ) \subset  B ( y_j, \kappa' d' ( y_j ) ). $   From disjointness of    $ \{ B ( x_m,
       \frac{1}{20} d (x_m)  ) \} $    and a  volume type argument we deduce that each  $ j $   corresponds to at most  $ \kappa'' $ integers $ m \in S_3  $ where $ \kappa'' $ is independent of $ j. $
 From   this fact, (\ref{3.10}), and disjointness of  $  \{ B ( y_j , \frac{1}{100}  d' (y_j)  )  \} $  we conclude that
   there is  a  $ \ti \kappa,  0 < \ti \kappa <  \infty, $  with
   \begin{equation}
   \label{3.15}
   K_2  \, \leq \,  \ti \kappa  M \sum_{m \in S_2}  d (x_m)^{n-1}   \leq
       \ti \kappa^2   M \sum_{j}  d' (y_j)^{n-1}    \,      \leq \ti \kappa^3 M H^{n-1} ( \ar D )  . 
       \end{equation}
       Using  (\ref{3.13})-(\ref{3.15}) in (\ref{3.12}) we find that Lemma \ref{lemma3.3}  is  valid.
       \qed
  \end{proof}
  Recall that  $ \nabla u' $  is  H\"{o}lder  continuous in $ \bar D \sem \La. $   
  We use this recollection and  Lemmas \ref{lemma3.2}, \ref{lemma3.3}, to prove
  \begin{lemma}  
  \label{lemma3.4}
  There exists $ c = c (p, n ) $  such that
    \[
    \int\limits_{ \ar D }  | \nabla u' |^{p-1}  | \log | \nabla u' | | \, d H^{n-1}  \, \leq \, c  \log M.   
    \]
    \end{lemma}
\begin{proof} 
From smoothness of $ u' $ in $ \bar D \sem \La, $ (\ref{1.2}), and integration by parts, we see that
        \begin{equation}
        \label{3.16}  
        d \mu' / dH^{n-1} =  | \nabla u'  |^{p-1} > 0 \mbox{ on }  \ar \Om' \sem \La.  
        \end{equation}

    We claim for some $ c = c ( p, n ) \geq 1   $   that
   \begin{equation}   
   \label{3.17}
   \frac{1}{c} \leq  \mu' ( \ar \Om'  \cap B ( 0,  10  ))   \leq \mu'  ( \ar \Om' ) \leq  c.  
   \end{equation}   
   To prove the left hand inequality in  (\ref{3.17})  we first observe from $ u ( \ti z )  = 1 $ and  
   Lemmas \ref{lemma2.1}, \ref{lemma2.2}, and (\ref{3.7}) 
   that  $ c^* u' \geq 1 $ on  $ \ar  B  ( \ti z, 4 r_1 ) $  for some  $ c^* = c^* (p,n) \geq 1. $  
   Let  $ l $ denote the line from the origin through $ \ti z $ and  let
   $ \ze_1 $  be the point on this line segment in  $  \ar B ( \ti z, 4 r_1 )  \cap B ( 0, 10). $  
   Let  $ \ze_2  $ be the point on  the line segment from $ \ze_1 $ to the origin with  $ d ( \ze_2, \ar \Om' )  =  \frac{1}{20}  r_1 $
    while  $ d ( \ze, \ar \Om' )  >   \frac{1}{20}  r_1 $ at every other point on  the line segment from
    $ \ze_1 $  to $ \ze_2. $  Then from  (\ref{3.4}), Lemma \ref{lemma2.1}, 
    and the above discussion we see that $ c^{**} u ( \ze_2 ) \geq  1 $ for some $ c^{**} (p, n ) \geq  1. $   
    Also,  $  B ( \ze_2,  \frac{1}{2}  r_1 ) \subset B (0, 10). $  Let  $  \hat \ze $ 
    be the point in $ \ar \Om' $ with  $  | \hat \ze - \ze_2 |  = d ( \ze_2, \ar \Om' ). $   
    Applying Lemma \ref{lemma2.3} with   $w =  \hat \ze,   r = 2 d ( \ze_2, \ar \Om' ), $ 
    we deduce that the left hand inequality in  (\ref{3.17}) is valid.  
    The  right hand  inequality  in this claim follows  once again from Lemma \ref{lemma2.3} and $ u' \leq u. $

    Let
    \[
\log^+ t  =  \max \{ \log t,  0  \}
    \]
and
\[
\log^- t =  \log^+ (1/t)
\]
for  $ t  \in (0, \infty). $   From Lemma \ref{lemma3.2}, (\ref{3.16}),  (\ref{3.17}),  and $ H^{n-1} ( \La ) = 0 $  
we  obtain for some $ c = c ( p, n) \geq 1, $
     \begin{equation}
     \label{3.18}  
     \int\limits_{\ar \Om'}
     |\nabla u' |^{p-1}  \, \log^+ | \nabla u' | \, dH^{n-1}
     \, \leq c  \, \log M \, \,  \mu' ( \ar \Om' ) \, \leq c^2  \log M.  
     \end{equation}

To estimate  $ \log^{-} | \nabla u' |,  $  fix $ \eta,  - \infty  \leq \eta  \leq - 1, $ and
  let  $ v' (x)  = \max \{ \log | \nabla u' |, \eta \} $ when  $ x \in \bar D \sem \La. $   Given   a  small
   $ \he > 0  $  let
  \[   \La ( \he ) = \{ x  \in  D : d ( x, \La )  \leq \he \}
  \, \, \mbox{and}\, \,  D ( \he ) =  D \sem  \La (\he).  \]
  From Lemma \ref{lemma2.4}  and  Lemmas \ref{lemma3.2}, \ref{lemma3.3} 
  we deduce that   $  |\nabla u' |^{p-2}  u'_{x_i } $ has a  $  W^{1,2} ( D ( \he ) ) $ 
  extension with distributional derivative $  ( |\nabla u' |^{p-2} u'_{x_i} )_{x_j}  = 0 $  when
   $ | \nabla u' | = 0 $  and $ 1 \leq i, j \leq n. $   Moreover  these functions are  continuous
  near  $ \ar D ( \he ) $  thanks to Lemmas \ref{lemma2.4} and \ref{lemma2.5}.  
  Let   $ \{ b_{ik} \}, L,   $  be  as defined  in  (\ref{1.3}), (\ref{1.4})  relative to  $u' $  and  note  from the above discussion that
   \[   
   L u'  (x)  =  (p-1)  \nabla \cdot \left( |\nabla u' |^{p-2}  \nabla u' \right) ( x ) = 0  
   \]
   exists pointwise for almost every $ x \in D ( \he ). $ Put
  \begin{equation}  
  \label{3.19}    
  I (\he)  =  \int\limits_{D(\he)}  Lu'  \, v'  \, dx + \int\limits_{D(\he)}   \sum_{i,k= 1}^n b_{ik}  u'_{x_k} \, v'_{x_i}  dx  =   I_1 ( \he ) + I_2 ( \he ) .  
  \end{equation}

        Clearly   $ I_1 (\he )  = 0. $    To handle  $ I_2 ( \he ) $  we first argue as  in  (\ref{3.8}), i.e, use a  barrier argument,   
        and  second  use  Lemma \ref{lemma2.5} to deduce for some $  c = c ( p, n ) \geq 1, $ that if
  $ r_2 = ( 1 + c^{-1} ) r_1, $ then
     \begin{equation}
     \label{3.20}
     \frac{1}{c} \, \leq \,  | \nabla u'  | \,  \leq c\, \, \mbox{on}\, \,  \bar B ( \ti z, 2 r_2 ) \sem  B ( \ti z,  2 r_1).  
     \end{equation}
Let  $ \psi  $  be  infinitely differentiable and $0\leq \psi \leq 1$ on  $ \rn{n} $  
with  $ \psi \equiv 1 $ on  $ \rn{n}  \sem  B ( \ti z,  2 r_2 ) $ and
$ | \nabla \psi | \leq c r_1^{-1} \leq c^2 ,$
  where the last inequality follows from  (\ref{3.4}) and the definition of  $ r_1. $    
  Suppose  also  that  $ \psi  $ vanishes in an open  set containing   $ \bar B ( \ti z,  2 r_1 ).  $     Then
   \begin{equation} 
   \label{3.21}
   \begin{array}{ll}  
   I_2 (\he)  &=    {\ds \int\limits_{D(\he)}   \sum_{i,k= 1}^n b_{ik}  (\psi  u' )_{x_k} \, v'_{x_i}  dx  }  +  { \ds \int\limits_{D(\he)}   \sum_{i,k= 1}^n b_{ik}  ( (1 - \psi)  u' )_{x_k} \, v'_{x_i}  dx }  \\ \\
   &   =  I_{21}  ( \he ) + I_{22} ( \he ).   
   \end{array}  
   \end{equation}

  From Lemmas \ref{lemma2.4}, \ref{lemma2.5}, (\ref{3.20}), and an argument similar to the one in (\ref{3.10})
we  deduce  for some  $ c = c (p,n) \geq 1 $ that
\begin{equation}
\label{3.22}
| I_{22} |  \leq  c.   
\end{equation}

  Turning to   $ I_{21} (\he )  $ we note  from Lemmas \ref{lemma3.2} and \ref{lemma3.3} 
  that the integrand in the  integral defining  $ I_{21} (\he ) $ is dominated  
  by an  integrable function independent of $ \he. $  Thus from the  Lebesgue dominated convergence theorem,
  \begin{equation}
  \label{3.23}  
  \lim_{ \he \to 0}  I_{21} (\he )   =
  \int\limits_{D}   \sum_{i,k= 1}^n b_{ik}  (  \psi  u' )_{x_k} \, v'_{x_i}  dx  = I' .  
  \end{equation}
    We assert that  
    \begin{equation}  
    \label{3.24}
    I'  \leq 0 . 
    \end{equation}
  To verify this assertion  let  $ u''  =  u'' (\de ) =  \max ( u' - \de, 0 ).  $  
  Using the convolution of  $  \psi u'' $ with an approximate identity and taking limits we see from  Lemma \ref{lemma2.6}   that
  \[ 
  \int\limits_{D}   \sum_{i,k= 1}^n b_{ik}  (  \psi  u'' )_{x_k} \, v'_{x_i}  dx  \leq 0 \, .    
  \]
  Now again  from  Lemmas \ref{lemma3.2} and \ref{lemma3.3}, 
  we observe that the above integrand is  dominated by an integrable 
  function independent of $  \de. $  Using this fact, the above inequality, and the 
  Lebesgue dominated convergence theorem  we  get  assertion  (\ref{3.24}). 
  Using  (\ref{3.19}) - (\ref{3.24}) we conclude (since $ I_{22} (\he ) $ is  independent of $\he$) that
         \begin{equation}
         \label{3.25} 
         \lim_{\he \to 0} I (\he )   \leq  c. 
         \end{equation}
On the other hand  from  \cite[Chapter 5]{GE92} and the discussion above (\ref{3.19}) we see that  integration by parts can be used to get
\begin{equation}  
\label{3.26}   
I_1 ( \he )  = - I_2 (\he )  +
\int\limits_{ \ar D (\he )  } v'  \,\sum_{i,k = 1}^n  \,  b_{ik} \,  u'_{x_k} \nu_i dH^{n-1}    
\end{equation}
where $ \nu = (\nu_1, \dots \nu_n ) $ is the outer unit normal to  $ \ar D ( \he ). $    
From (\ref{3.20})   we see that
  \begin{equation}
  \label{3.27}  
  \left|\, \int\limits_{ \ar B ( \ti z, 2 r_1 ) }
 v'  \,\sum_{i,k = 1}^n  b_{ik} \,  u'_{x_k} \nu_i dH^{n-1}  \right|  \leq  c = c (p,n).    
 \end{equation}
 From  Lemma \ref{lemma3.2},  dominated  convergence,  and the definition of  $ D (\he ), $  we have
 \begin{equation}  
 \label{3.28} 
 \int\limits_{ \ar D ( \he) \sem \ar B ( \ti z, 2 r_1 ) }
 v'  \,\sum_{i,k = 1}^n  b_{ik} \,  u'_{x_k} \nu_i dH^{n-1}   \to   \int\limits_{ \ar \Om' \sem \La  }   v' 
 \,\sum_{i,k = 1}^n  b_{ik} \,  u'_{x_k} \nu_i dH^{n-1}  \mbox{ as } \he \to 0.  
 \end{equation}

 Observe that $ \nu  =   - \frac{\nabla u'}{|\nabla u' |} $ on  $ \ar \Om'\setminus\Lambda. $
  From this observation and  (\ref{1.4})  we calculate
 \begin{equation} 
 \label{3.29}
 \begin{array}{ll}  
   \, {\ds \sum_{i,k = 1}^n  b_{ik} \,  u'_{x_k}}  \nu_i & = -
 {\ds \sum_{i,k = 1}^n | \nabla  u' |^{ p - 5} [ ( p - 2 )
    (u')_{x_i}^2   (u')_{x_k}^2  + \de_{ik}
    | \nabla  u' |^2 ]  u_{x_i}  \, u_{x_k} } \\ \\  &  =  -  (p-1) | \nabla u' |^{p-1} .  
    \end{array}  
    \end{equation}

From (\ref{3.19}), (\ref{3.25})-(\ref{3.29}) we find that
\begin{equation}
\label{3.30} 
- (p-1)  \int\limits_{ \ar \Om' } v \,  |\nabla u' |^{p-1}  \,   \, d H^{n-1}
\leq  \lim_{\he \to 0}  I ( \he )  +  c \,   \leq \, 2 c.  
\end{equation}
Letting  $ \eta \to - \infty $ in  (\ref{3.30}) and using  the monotone convergence theorem we see that
(\ref{3.30})  holds with $ v $ replaced by  $ \log | \nabla u |. $  Finally from  (\ref{3.30}) for
$ \log |\nabla u |  $  and  (\ref{3.18})  we conclude the validity of Lemma \ref{lemma3.4}.  
\qed
\end{proof}

With these lemmas in hand, we go back to the proof of (\ref{3.3}) and Proposition \ref{proposition3.1}b.  We note  from Lemma \ref{lemma2.3} and $ u' \leq u $  that for given  $ j, 1 \leq j \leq N, $
\begin{equation}
\label{3.31}
t_j^{1-n}\,  \mu' ( \bar{B} ( z_j, t_j ) )  \leq  c \,  t_j^{1-p}\,  \max_{B(z_j , 2t_j )} u^{p-1} \leq  c^2\,    t_j^{1-n}\,   \mu (  B ( z_j, 4 t_j) ) .     
   \end{equation}
  For  given  $ A  > > 1,  $  we see from (\ref{3.6}) that  $ \{ 1, 2, \dots, N \} $ can be divided into disjoint subsets $ \Ph_1, \Ph_2,
  \Ph_3, $ as follows.
  \begin{displaymath}
   \left\{
   \begin{array}{ll}
    j\in \Ph_ 1 & \mbox{if}\, \, t_j >  s, \\
    j \in \Ph_2 & \mbox{if}\, \, t_j = s\, \, \mbox{and}\, \, | \nabla u' |^{p-1}  (x) \geq  M^{- A },
 \, \mbox{for some}\, \,  x  \in  \ar \Om' \cap \ar  B ( z_j, t_j ) \sem \La \\
 j\in \Ph_3 & \mbox{if $j$ is not in $ \Ph_1 $ or $ \Ph_2$.}
   \end{array}
   \right.
  \end{displaymath}
        Let  $ t_j'  = t_j $ when $ j \in \Ph_1 $ and $ t_j' = 4s  $ when $ j \in \Ph_2. $
    To prove (\ref{3.3}) set
   \[
   E  = \ar \Om  \cap  \bigcup_{j\in \Ph_1 \cup \Ph_2}  B ( z_j, t_j' ). 
   \]
   To estimate $ \ph^{\la}_{\tau} ( E ) $ we first observe that if
        \begin{equation}
        \label{3.32}   
        x \in \bigcup_{j\in \Ph_1 \cup \Ph_2}  B ( z_j, t_j' )
        \mbox{ then  $ x $ lies in at most $ c = c ( n ) $ of
   $ \{ B ( z_j, t_j' ) \}. $ }   
   \end{equation}
     This observation can be proved using $ t_j \geq  s, 1 \leq j \leq N, $
   a volume type  argument,  and  the fact that
   $ \{ B ( z_j, t_j)\}_1^N  $ is a  Besicovitch covering of  $ \ar \Om \cap \bar B (0, 15). $
   If $ j \in \Ph_2 $  we get  from  (\ref{3.8}),  (\ref{3.31}),  that for some $ c = c (p, n ) \geq 1 $
   \[   
   M^{-A} \leq | \nabla u' ( x ) |^{p-1}  \leq  c \,   s^{1-n}   \mu (  B ( z_j, 4 s) ) \, . 
   \]
      Rearranging this inequality, summing, and using (\ref{3.1}), (\ref{3.32}), we see that
        \[
        \sum_{j\in \Ph_2} (t_j')^{n-1} \,  \leq  \, \ti c  \, M^A \mu ( \bigcup_{j \in \Ph_2}   B ( z_j, t_j' ) )
        \leq (\ti c)^2 \,  M^A  
        \]  
        provided $ \ti c = \ti c (p, n ) $ is large enough.  Now since $ t_j' = s $ for all
    $  j \in \Ph_2 $ we may for given $ A, M, \ep $  choose $ s > 0 $  so small  that
    \begin{equation}  
    \label{3.33}  
    s^{1-n}   \la ( s )   \leq  \frac{\ep}{ 2 (\ti c)^2 M^A } 
    \end{equation}
where we have used the definition of $\lambda$.
    Using this choice of  $ s $ in the above display  we get
    \begin{equation}
    \label{3.34}    
    \sum_{j\in \Ph_2} \la (t_j') \,  \leq  \,  \ep/2.  
    \end{equation}
     On the other hand  we may suppose  $ \bar \tau $ in  (\ref{3.5})   is so small that
     $ \la ( t_j ) \leq  t_j^{n-1} $ for $ 1 \leq j \leq N. $   Then from (\ref{3.1}), (\ref{3.6}), and (\ref{3.32}), we see that
  \begin{equation} 
  \label{3.35}   
  \begin{array}{ll}
  {\ds  \sum_{j \in  \Ph_1}  \la(t_j') } &  \leq
   { \ds \sum_{ j \in  \Ph_1} } (t_j')^{n-1} \\ \\ &  =    M^{-1}  {\ds \sum_{j \in \Ph_1} } \mu ( B ( z_j, t_j ) ) \leq  \ep/2  
   \end{array}
  \end{equation} 
  provided $ M = M ( \ep ) $ is chosen large enough.
    Fix $ M $ satisfying all of  the above requirements.  In view of  (\ref{3.34}), (\ref{3.35}),  
    we have  proved  the  left hand  inequality in (\ref{3.3})
for $ E $ as defined above, i.e. $\phi^{\la}_\tau (E) \leq \epsilon$.

To prove the right hand inequality in  (\ref{3.3}) we use Lemma \ref{lemma3.4}  and  the definition of  $ \Ph_3 $ to obtain
  \begin{equation} 
  \label{3.36}
  \begin{aligned}
     \mu' \left(  \ar \Om' \cap  \bigcup_{j\in \Ph_3}  \bar{B} ( z_j, t_j)
  \right)  & \leq  \mu' \left( \{ x \in \ar \Om' :  | \nabla u' (x) |^{p-1}  \leq M^{-A} \} \right) \\
    & \leq    (p-1) ( A  \log M )^{-1}  {\ds \int_{\ar \Om' } } | \nabla u' |^{p-1}  \, | \log |\nabla u' | |  d H^{n-1} \\
    &\leq \frac{c}{A}.
    \end{aligned}
     \end{equation}
    Choosing  $ A = A ( n ) $ large enough  we have from     (\ref{3.17}),   (\ref{3.36}),
  \begin{equation} 
  \label{3.37}
  \mu' \left(    \bigcup_{j\in \Ph_1 \cup \Ph_2}  B (0, 10) \cap  \bar{B} ( z_j, t_j)
  \right)  \geq  \mu' ( B (0, 10 ) ) - \mu' \left(   \bigcup_{j\in \Ph_3}  \bar{B} ( z_j, t_j)
  \right) \geq   c_*^{-1}
     \end{equation} 
     for some $ c_* (p,n). $
    Finally from (\ref{3.31}), (\ref{3.32}), and (\ref{3.37}), we get for some $ c = c(p, n ) \geq 1 $ that
    \begin{equation} 
    \label{3.38}
    \mu( E )   \geq c^{-1}    \sum_{j\in \Ph_1 \cup \Ph_2} \, \mu ( \bar{B} ( z_j, t'_j)
  )  \geq  c^{-2}  \sum_{j \in \Ph_1 \cup \Ph_2} \mu' ( \bar{B} ( z_j, t_j) )   \geq c^{-3}.   
  \end{equation}
For $j\in \Phi_1$ we have used the definition of $t_j$ so that 
\[
\mu( B(z_j , 4t_j)) < M 4^{n-1} t_j^{n-1} =4^{n-1} \mu( B(z_j,t_j)) =4^{n-1} \mu(B(z_j,t'_j)) 
\]
      Thus  (\ref{3.3}) is valid.  Proposition \ref{proposition3.1} follows from (\ref{3.3})  and  our earlier  remarks.  
      \qed
      \end{proof}
    \subsection{Proof of  Theorem \ref{theorem1.1}}   
    \label{subsection3.2}
    Next we show for $ \la, Q $ as in Proposition \ref{proposition3.1}  that there exists a Borel set $  Q_1 $ with
    \begin{equation} 
    \label{3.39}   
    Q_1 \subset  Q , \,  \hat \mu ( \ar O \cap B ( \hat z, \rho) \sem Q_1 )
    = 0,  \, \mbox{  and  }  H^\la ( Q_1 ) = 0.  
    \end{equation}
     To prove  (\ref{3.39})  we assume, as we may, that
     $\hat  \mu ( \ar O \cap B ( \hat z, \rho) ) <  \infty $ since  otherwise we can write  $ \ar O  \cap B ( \hat z, \rho ) $ 
     as a countable union of Borel sets with
finite $ \hat \mu $ measure and apply the following argument in each set.  Under this assumption we can use
Proposition \ref{proposition3.1} and a  Vitali type covering argument (see \cite{Ma95}), as well as  induction
to get compact sets $ \{F_l\},  \,  F_l  \subset  Q, $   with  $ F_k \cap F_j = \es, \, k \not = j,  $
    $ \hat \mu ( F_1 ) > 0 $  and with
    \[  
    c'   \hat \mu ( F_{m+1}  )  \geq   \hat  \mu ( Q \sem \bigcup_{l = 1}^m  F_l ),  m = 1, 2, \dots, 
    \]
    for some $ c' = c'  (p, n ) \geq 1. $   Moreover
    $ H^{\la} ( F_l ) =  0 $ for  all $l.$  Then   $ Q_1  =  \bigcup_{l=1}^\infty  F_l  $  
    has the desired properties as follows from measure theoretic arguments.

     To prove
    Theorem \ref{theorem1.1} we first  note from a  covering argument as in  \cite{LNP11} or \cite{W93} that if
    \[ 
    P =  \{ x \in \ar O \cap B ( \hat z, \rho ) :  \limsup_{t \to 0}\, \,  \frac{
     \hat \mu ( B ( x, t ) )}{t^{n-1}}  >  0  \},  
     \]
     then $ P $ has $\si$  finite $ H^{n-1} $ measure.
    For completeness we  prove this statement after finishing the proof of  Theorem \ref{theorem1.1}.
    Thus  to prove Theorem \ref{theorem1.1}  it suffices to  show that
    \begin{equation} 
    \label{3.40}  
    \hat \mu ( Q_1 \sem P ) = 0.  
    \end{equation}   
    Indeed otherwise from Egoroff's theorem there exists a compact set $ K \subset  Q_1 \sem P $ with
    \begin{equation}  
    \label{3.41}  
    \hat \mu ( K ) > 0\, \,  \mbox{and}\, \, {\ds \lim_{t \to 0} } \, \, \frac{\hat \mu ( B ( x, t ) )}{t^{n-1}} = 0 \,\, \mbox{uniformly for}\, \, x \in K. 
    \end{equation}  
    Choose $ \al_k \in (0, 1), k =  1, 2, \dots, $   with
    $ \al_{k+1} <  \al_k/2 $ and so that
    \[  
    \sup_{0  < t \leq \al_k }    \frac{\hat \mu ( B ( x, t ) )}{t^{n-1}}  \leq 2^{-2k} \mbox{ for all } x \in K. 
    \]
    Let $ \al_0 = 1. $ With  $ (\al_k )_0^\infty  $ now chosen, define $ \la ( t ) $  on  $(0, 1]$ by
    $ \la ( \al_k ) =  2^{-k}  ( \al_k)^{n-1}, k = 0, 1, \dots, $  and  $ t^{1-n}  \la (t)  $ is linear for
    $ t $ in the intervals
     $ [\al_{k+1}, \al_k ] $  for $ k = 0, 1, \dots $   Put $ \la (0) = 0. $  Clearly   $ t^{1-n}  \la ( t ) \rar 0 $  as $ t \rar 0. $   Also,  if
     $ \al_{k+1} \leq t \leq \al_k, $  and  $ x \in K, $ then
     \begin{equation}  
     \label{3.42}
     \frac{ \hat \mu ( B ( x, t ) }{ \la (t) } \leq   2^{1 - k}. 
     \end{equation}
                     Given $ m $ a positive integer  we note from (\ref{3.39}) 
                     that there is a covering  $ \{ B ( x_j, r_j ) \} $ of  $ K $ with $ r_j \leq \al_m / 2$ for all $ j $ and
    \[
    \sum_j  \la ( 2 r_j  )  \, \leq \,  1  
    \]
    We may assume that                  there is  an  $ x_j' \in  K \cap B ( x_j, r_j ) $ for each $ j $ 
    since otherwise we discard  $ B ( x_j,  r_j ). $   Moreover  from  (\ref{3.42})  we see that
    \[  
    \hat \mu ( K )  \leq  \sum_j  \hat \mu ( B ( x_j',  2 r_j ) ) \leq  2^{1 - m}    \sum_j  \la (2 r_j )  \leq  2^{1-m}  .  
    \]
    Since $m$ is  arbitrary  we have reached a  contradiction to  $ \hat \mu (K) > 0 $ in (\ref{3.41}). From this contradiction
    we conclude  first  (\ref{3.40})  and second  Theorem \ref{theorem1.1}.
    \qed

To prove that $ P $ has $ \si $ finite $ H^{n-1} $ measure
   we  once again may assume   $ \hat \mu ( \ar O  \cap B ( \hat z , \hat \rho ) )  < \infty . $
     Let
     \[
     P_m =  \{ x \in  P  :  \limsup_{t \rar 0}   t^{ 1 -n}
    \hat \mu ( B ( x, t ) )  >  \frac{1}{m} \} 
    \]
    for $ m = 1, 2, \dots $
                                              Given $ \de > 0 $  we choose a Besicovitch covering
                                               $ \{ B ( y_i,  r_i ) \}
    $   of $  P_m   $ with $ y_i \in P_m, \,
    r_i \leq \de, B ( y_i,  r_i ) \subset  B ( \hat z, \rho ) $   and  
    \[
    \mu ( B ( y_i,  r_i ) ) \,  >  \,
    \,    \frac{r_i^{n-1}}{m} .  
    \]
          Thus
      \begin{equation}
      \label{3.43}
     \sum_{i} \, r_i^{ n - 1}  <  m \sum_{i} \hat \mu ( B ( x_i,  r_i ) )
             \leq  c \, m
       \, \hat \mu    (\ar O \cap B ( \hat z,  \rho ) )  < \infty.   
       \end{equation}
                   Letting $ \de  \rar 0 $ and using the definition of $
    H^{n - 1}  $ measure  we conclude from (\ref{3.43})  that $ H^{n - 1}  (  P_m )  < \infty. $
     Hence  $ P $ has $ \si $ finite  $ H^{n-1} $ measure.
\section{Closing Remarks}
 \label{section4}
 The existence of a measure,  say $  \mu, $   corresponding to a  positive weak
solution $ u $ in $ O \cap B ( \hat z, r ) $  with vanishing boundary
values, as in  (\ref{1.2}), can be  shown  for a large
class of  divergence form partial differential equations.
  What can be said about $ \hd  \mu? $  What can be said
  about  analogues of  Theorems  \ref{theorem1.1}, \ref{theorem1.3}?   Regarding these questions we note that
  Akman in \cite{A13} has considered  PDE's   whose  Euler equations  arise from   minimization problems  with integrands involving  $ f ( \nabla v ) $ and   $ v  \in  W^{1,p}.$
  More specifically for  fixed
  $ p, 1 < p < \infty, $  the function   $f:\rn{2} \sem \{0\} \to(0,\infty),$  is  homogeneous of degree $p$ on $\rn{2}$. That is,
\[  
f(\eta)=|\eta|^{p}f\left(\frac{\eta}{|\eta|}\right)>0\,\ \mbox{when}\,\ \eta
= (\eta_1, \eta_2)  \in\rn{2} \sem \{0\} .
\]  
Also  $ \nabla f = ( f_{\eta_1}, f_{\eta_2} ) $ is   $ \de $ monotone on   $\rn{2} $ for some $ \de > 0$
(see \cite{AIM09} for a definition of $ \de $ monotone).        
In \cite{A13}, Akman   considers  weak solutions to the Euler-Lagrange equation,
\begin{equation} 
\label{4.1} \sum\limits_{k=1}^{2}\frac{\ar}{\partial x_{k}}\left(\frac{\ar f}{\partial\eta_k}(\nabla u(x))\right) = 0 \, \mbox{ when } x = (x_1, x_2)  \in \Om \cap N, 
\end{equation}  
where $ \Om \subset  \rn{2} $ is a bounded simply connected domain and $ N $ is a neighborhood of  $ \ar \Om. $
   Assume also that $ u > 0 $ is continuous in $ N $ with $ u \equiv 0 $ in $ N  \sem \Om. $ Under these assumptions it follows that there exists
a unique finite  positive  Borel measure $ \mu $ with support in $ \ar \Om $ satisfying
\[
\int\limits_{\rn{2}}\lan\nabla f(\nabla u), \nabla\phi\ran dA=-\int\limits_{\partial\Omega}\phi d \mu
\] 
whenever $\phi\in C_{0}^{\infty}(N)$.
He proves

\begin{theorem} 
\label{theorem4.1}
Let  $p, f,  \Om, N,  u, \mu $  be as above   and put
\[
\la(r)=r\,\exp   \left[ A\sqrt{\log\frac{1}{r}\,\log\log\frac{1}{r}} \right]\,\ \mbox{for}\, \ 0<r<10^{-6}.
\]        
\[ 
\begin{array}{l}  
(a) \hs{.2in}  \mbox{ If $p\geq 2$, there exists $A=A(p)\leq -1$ such that $ \mu $ is }
 \\ \hs{.38in} \mbox{ concentrated on a set of $\si-$finite $H^{\la}$ Hausdorff measure.}\\
(b) \hs{.2in} \mbox{  If  $1<p\leq 2$, there exists $A=A(p)\geq 1$, such that $\mu$ is absolutely }  \\ \hs{.38in}  \mbox{continuous with respect to $H^{\la}$ Hausdorff measure. } \end{array} \]
\end{theorem}
  For  $ p  = 2 $  and  $ f ( \eta ) = | \eta |^p $ the above theorem is slightly weaker than
 Theorem \ref{theorem1.3}.    It is easily seen that  Theorem \ref{theorem4.1} implies
 \[  
 \hd \mu  \leq  1  \mbox{ for }  p  \geq 2  \mbox{ and }  \hd \mu \geq 1 \mbox{ for } 1 < p \leq  2. 
 \]
  A key argument in the proof of  Theorem \ref{theorem4.1}  involves
  showing that  $ \zeta =  \log f ( \nabla u ) $ is a weak subsolution, supersolution or solution to
  \[
  L \zeta (x)   =  \sum_{k,j=1}^2   \frac{\ar}{\ar x_k} \left( f_{\eta_k \eta_j}(\nabla u(z))  \frac{\ar \zeta (x)  }{ \ar x_j } \right)  \mbox{ when }   x \in  \Om \cap N 
  \]
  and  $ p > 2, 1 < p < 2, $
  $ p = 2, $ respectively. In  \cite{ALV12} this was shown pointwise at
  $ x \in \Om \cap N $ when $ \nabla u, f, $ are sufficiently smooth
   and $ \nabla u ( x ) \not = 0. $  We plan  to use this  fact   and  the technique in  Theorem \ref{theorem4.1}
    to   prove  analogues of  Theorem \ref{theorem4.1}  when  $ n = 2 $ and also  higher dimensional analogues.
  The case $ p = n $ in Theorem \ref{theorem1.1} and $ p = 2 $ in  the proposed generalization of  Theorem \ref{theorem4.1} are 
particularly interesting.  Can one for example do away with  the uniform fatness assumption in  
Theorem \ref{theorem1.1}  or  the proposed generalization of   Theorem \ref{theorem4.1}  when $p = 2, n = 2?$  
The argument in \cite{W93}   and  \cite{JW88} relies on  a  
certain integral inequality (see Lemma 3.1 in \cite{JW88}).
\bibliographystyle{spmpsci}
\bibliography{myref}   
\end{document}